\documentclass[11pt, a4paper]{article}
\usepackage{amsfonts}
\usepackage{amsmath,amssymb,amsthm,amsfonts,enumerate}
\allowdisplaybreaks[4]
\usepackage{graphicx}
\usepackage{float}
\usepackage[scriptsize,nooneline]{subfigure}
\usepackage{indentfirst}
\usepackage{cite}
\usepackage{color}
\usepackage{mathrsfs}
\usepackage{epstopdf}
\usepackage{setspace}
\usepackage[labelfont=footnotesize,font=footnotesize]{caption}
\usepackage[colorlinks, citecolor=red]{hyperref}
\usepackage{esint}

\newtheorem{theorem}{Theorem}[section]
\newtheorem{lemma}{Lemma}[section]
\newtheorem{definition}{Definition}[section]

\newtheorem{remark}{Remark}[section]
\newtheorem{corollary}{Corollary}[section]
\def\geq{\geqslant}\def\leq{\leqslant}
\date{}
\begin{document}

\title{\bf Generalized Campanato Space Over Non-homogeneous\\ Space and Its Applications\\[10pt]}
\author{Yuxun Zhang, Jiang Zhou\thanks{Corresponding author. The research was supported by National Natural Science Foundation of China (Grant No. 12061069).}  \\[.5cm]}

\maketitle
{\bf Abstract: \rm The authors introduce generalized Campanato space with regularized condition over non-homogeneous space, and study its basic properties including the John-Nirenberg inequality and equivalent characterizations. As applications, the boundedness of fractional type Marcinkiewicz integral operator and its commutator on generalized Morrey space over non-homogeneous space is obtained.

{}\par



\baselineskip 16.3pt

\section{Introduction}
In 1961, John and Nirenberg\cite{Joh61} first introduced the space $\mathrm{BMO}(\mathbb{R}^n)$. A function $f\in L^1_{loc}$ is in $\mathrm{BMO}(\mathbb{R}^n)$ if
\begin{equation}
\|f\|_{\mathrm{BMO}(\mathbb{R}^n)}:=\sup_{B}\fint_B|f(x)-f_B|dx<\infty,
\end{equation}
where the supremum is over all balls $B\subset\mathbb{R}^n$, $f_B:=\fint_Bf(y)dy$ denotes the mean value of $f$ on $B$.

Campanato\cite{Cam63, Cam64} introduced Campanato space and researched its fundamental natures in 1963 and 1964. Since then, for applications to the regularity of solutions of partial differential equations (e.g. in\cite{Gia83}), Campanato space has been further studied and gradually generalized by many authors, see\cite{Pee69,Nak85,Nak06}.

In 1985, Nakai and Yabuta\cite{Nak85} introduced the generalized Campanato space over $\mathbb{R}^n$ as follow: For $1\leq p<\infty$ and a measurable function $\varphi:\mathbb{R}^n\times(0,\infty)\rightarrow(0,\infty)$, the generalized Campanato space $\mathcal{L}^{p,\varphi}(\mathbb{R}^n)$ is consisted of all $f\in L_{loc}^1$ such that
\begin{equation}
\|f\|_{\mathcal{L}^{p,\varphi}(\mathbb{R}^n)}:=\sup_B\left(\frac{1}{\varphi(B)}\fint_B|f(x)-f_B|^pdx\right)^{\frac{1}{p}}<\infty,
\end{equation}
where $\varphi(B):=\varphi(x,r)$ for ball $B=B(x,r)\subset\mathbb{R}^n$. Note that if $\varphi(B)=1$ for all balls $B$, the condition (2) will equal to (1)\cite{Joh61}, so $\mathcal{L}^{p,\varphi}(\mathbb{R}^n)$ is surely the generalization of $\mathrm{BMO}(\mathbb{R}^n)$.

In 2010, in order to incorporate some spaces whose measures dissatisfies the doubling condition into a unified framework for studying, Hyt\"{o}nen\cite{Hyt10} introduced a new type of metric measure spaces called non-homogeneous space, which satisfy the geometrically doubling and upper doubling conditions. Many works about function space and operator theory over $\mathbb{R}^n$ have been generalized to non-homogeneous space, see\cite{Cao13,Fu14,Lu17}.

Hyt\"{o}nen\cite{Hyt10} also introduced $\mathrm{RBMO}(\mu)$ space over non-homogeneous space, which is the analogue of $\mathrm{BMO}(\mathbb{R}^n)$ with an additional regularized condition. In 2014, Fu, Yang and Yang\cite{Fu14} introduced $\widetilde{\mathrm{RBMO}}(\mu)$ space over non-homogeneous space via the discrete coefficient, and obtained the boundedness of Calder\'{o}n-Zygmund operators on $\widetilde{\mathrm{RBMO}}(\mu)$. Subsequently, some results for the commutators generated by several operators and $\widetilde{\mathrm{RBMO}}(\mu)$ functions are obtained, for example, in\cite{Lin17,Lu21,Lu22}. Fu, Lin, Yang and Yang\cite{Fu15} introduced Campanato space over non-homogeneous space in 2015, which is the generalization of $\widetilde{\mathrm{RBMO}}(\mu)$.

In this article, we naturally introduce the generalized Campanato space over non-homogeneous space, and obtain some of their properties. Meanwhile, considering that the boundedness of operators is an important issue in function space theory (e.g. in\cite{Dai23}), as applications, the boundedness of Marcinkiewicz integral operator $\widetilde{\mathcal{M}}_{l, \rho, s}$ and its commutator $\widetilde{\mathcal{M}}_{l, \rho, s, b}$ from generalized Morrey space $L^{p,\phi}$ to $L^{q,\phi}$ is obtained, where $b$ is in generalized Campanato space. These results extend some theorems in\cite{Lu22} that showed the boundedness of $\widetilde{\mathcal{M}}_{l, \rho, s}$ and $\widetilde{\mathcal{M}}_{l, \rho, s, b}$ on $L^{p,\phi}$, and in\cite{Ku22} that studied the boundedness of these operators over $\mathbb{R}^n$.

Throughout this paper, we use $(\mathcal{X},d,\mu)$ to denote a non-homogeneous space unless there exists a contrary description, use $L^p_{loc}$ to denote the set of all $p$-th locally integrable functions for $p\in[1,\infty)$, use $L_{c}^{\infty}$ to denote the set of all $L^{\infty}$ functions supported on a certain ball, and use $\chi_{E}$ to denote the characteristic function of $E\subset\mathcal{X}$. We use $\mathbb{Z}$ to represent the set of all integers, and $\mathbb{N}:=\mathbb{Z}\cap[0,\infty)$. For any ball $B\subset\mathcal{X}$, $c_B$ and $r_B$ denotes the centre and radius of $B$, respectively, and for $\lambda>0$, $\lambda B:=B(c_B,\lambda r_B)$. For $\phi:\mathcal{X}\times(0,\infty)\rightarrow(0,\infty)$, $\phi(B)$ denotes $\phi(c_B,r_B)$. For any $p \in(1, \infty)$, we denote $p':=p/(p-1)$. Use $A\lesssim B$ to show $A\leq CB$, where $C>0$ is independent of the main parameters, and use $A\approx B$ to show $A\lesssim B$ and $B\lesssim A$. For some parameters $\alpha_1,\alpha_2,\cdots,\alpha_n$, $C_{(\alpha_1,\alpha_2,\cdots,\alpha_n)}$ or $c_{(\alpha_1,\alpha_2,\cdots,\alpha_n)}$ denotes a positive constant only dependent on $\mathcal{X}$ and $\alpha_1,\alpha_2,\cdots,\alpha_n$. For ball $B\subset\mathcal{X}$ and $f\in L^1_{loc}$, $m_{B}(f)$ is coincide with $f_B$.

\section{Preliminaries}

In this section, we give some basic definitions and lemmas. First, we recall some concepts about non-homogeneous space introduced by Hyt\"{o}nen.

\begin{definition}\cite{Hyt10}
A metric space $(\mathcal{X}, d)$ satisfies the geometrically doubling condition if there exists $N_{0}\in \mathbb{N}$ such that, any ball $B(x, r) \subset \mathcal{X}$ can be covered by at most $N_0$ balls $\left\{B\left(x_{i}, r / 2\right)\right\}_{i}$.
\end{definition}

\begin{definition}\cite{Hyt10}
A metric space $(\mathcal{X}, d)$ satisfies the upper doubling condition if $\mu$ is a Borel measure on $\mathcal{X}$, and there exists $\lambda: \mathcal{X} \times(0, \infty) \rightarrow(0, \infty)$ satisfying that $r \rightarrow \lambda(x, r)$ is non-decreasing for given $x \in \mathcal{X}$, and for all $x \in \mathcal{X}$, $r \in(0, \infty)$,
$$\mu(B(x, r)) \leq \lambda(x, r) \leq C_{(\lambda)} \lambda(x, r / 2) .$$
\end{definition}

\begin{remark}
In\cite{Hyt12}, the authors obtain that there exists $\widetilde{\lambda}$ pointwise controlled by $\lambda$, which makes $C_{(\widetilde{\lambda})}\leq C_{(\lambda)}$, and for all $x,y\in\mathcal{X}$ with $d(x,y)\leq r_0$,
\begin{equation}
\widetilde{\lambda}(x,r_0)\leq C_{(\lambda)}\widetilde{\lambda}(y,r_0).
\end{equation}
Therefore, we always assume that $\lambda$ satisfies (3).
\end{remark}

The following is the definition of discrete coefficient.

\begin{definition}\cite{Bui13}
Let $\tau \in(1, \infty)$, balls $B \subset S\subset\mathcal{X}$. Define
$$\widetilde{K}_{B, S}^{(\tau)}=1+\sum_{k=-\lfloor\log _{\tau} 2\rfloor}^{N_{B, S}^{(\tau)}} \frac{\mu(\tau^{k} B)}{\lambda(c_{B}, \tau^{k} r_{B})},$$
where $N_{B, S}^{(\tau)}$ is the smallest integer which makes $\tau^{N_{B, S}^{(\tau)}} r_{B} \geqslant r_{S}$, and $\lfloor x\rfloor$ denotes the greatest integer not more than $x$ for $x \in \mathbb{R}$.
\end{definition}

Though the doubling condition does not always hold for all balls $B\subset\mathcal{X}$, there still exist some balls satisfying the following property.

\begin{definition}\cite{Hyt10}
For $1<\alpha, \beta <\infty$, a ball $B \subset \mathcal{X}$ is called $(\alpha, \beta)$-doubling if $\mu(\alpha B) \leq \beta \mu(B)$.
\end{definition}

\begin{remark}
From \cite[Lemma 3.2]{Hyt10} and \cite[Lemma 3.3]{Hyt10}, set $\nu=\log _{2} C_{(\lambda)}$ and $n_{0}=\log _{2} N_{0}$, where $N_{0}$ is defined in Definition 2.1, for any $1<\alpha<\infty$ and ball $B\subset\mathcal{X}$, the smallest $(\alpha, \beta_{\alpha})$-doubling ball $\alpha^{i} B$ with $i \in \mathbb{N}$ exists, which is denoted by $\widetilde{B}^{\alpha}$, where
$$\beta_{\alpha}:=\alpha^{\max \left\{n_{0}, \nu\right\}}+30^{n_{0}}+30^{\nu} .$$
To simplify writing, we use $(\alpha,\beta)$-doubling ball to denote $(\alpha,\beta_{\alpha})$-doubling ball, and the doubling ball means the $(6,\beta_6)$-doubling ball.
\end{remark}

Next, to obtain the boundedness of certain operators, we introduce the set $\mathcal{G}_{\delta}^{dec }$ as the subset of $\mathcal{G}^{d e c}$ defined in\cite{Ara18}.

\begin{definition}
For $\delta \in(0,1)$, let $\mathcal{G}_{\delta}^{dec}$ be the set of all functions $\phi: \mathcal{X} \times(0, \infty) \rightarrow(0, \infty)$ such that for any $x\in\mathcal{X}$,
\begin{equation}
\lim_{r \rightarrow 0^{+}} \phi(x, r)=+\infty,\ \ \lim_{r \rightarrow+\infty} \phi(x, r)=0,
\end{equation}
and for any $\eta>1$ and balls $B_1\subset B_2$,
$$\phi(B_1)(\mu(\eta B_1))^{\delta} \geq c_{(\phi,\eta)}(B_2)(\mu(\eta B_2))^{\delta},\ \ \phi(B_1) \mu(\eta B_1) \leq C_{(\phi,\eta)} \phi(B_2) \mu(\eta B_2).$$
\end{definition}

\begin{remark}
It follows from\cite[Proposation 3.4]{Nak08} that, if $\phi \in \mathcal{G}_{\delta}^{d e c}$ satisfies (4), then there exists $\widetilde{\phi} \in \mathcal{G}_{\delta}^{d e c}$ equivalent to $\phi$, which is continous and strictly decreasing of $r$ for fixed $x$. Without the loss of generality, we still use $\phi$ to denote the corresponding $\widetilde{\phi}$.
\end{remark}

The following definition of generalized Morrey space is different from\cite{Lu17}, but similar to\cite{Nak94}.

\begin{definition}
Let $p \in[1, \infty)$, $\eta\in(1,\infty)$ and $\phi\in\mathcal{G}^{d e c}_{\delta}$. $f\in L^p_{loc}$ is in the generalized Morrey space $L^{p,\phi,\eta}$ if
$$\|f\|_{L^{p, \phi,\eta}}:=\sup _{B}\left(\frac{1}{\phi(B) \mu(\eta B)} \int_{B}|f(x)|^{p} d \mu(x)\right)^{\frac{1}{p}}<\infty.$$
\end{definition}

\begin{remark}
Similar to the proof in \cite[Theorem 7]{Cao13}, $L^{p, \phi,\eta}$ is independent of $\eta$. Therefore, $L^{p, \phi,\eta}$ can be written as $L^{p, \phi}$.
\end{remark}

Then, we introduce the generalized Campanato space $\widetilde{\mathcal{L}}^{\psi,\tau,\gamma}$.

\begin{definition}
Let $\tau \in(1, \infty)$ and $\gamma \in[1, \infty)$, $\psi:\mathcal{X}\times(0,\infty)\rightarrow(0,\infty)$ satisfy that, there exists $C>0$ such that for any two balls $B=B(x,r)$, $B'=B(x',r)$ with $d(x,x')\leq r$,
\begin{equation}
\psi(2B)\leq C\psi(B),\ \ \frac{1}{C}\psi(B)\leq\psi(B')\leq C\psi(B).
\end{equation}
$f\in L^1_{loc}$ is in $\widetilde{\mathcal{L}}^{\psi,\tau,\gamma}$ if there exists $C>0$ such that for any ball $B\subset\mathcal{X}$,
\begin{equation}
\frac{1}{\psi(B)} \frac{1}{\mu(\tau B)} \int_{B}\left|f(x)-f_{B}\right| d \mu(x) \leq C,
\end{equation}
and for any two balls $B \subset S$,
\begin{equation}
\frac{1}{\psi(B)}\left|f_{B}-f_{S}\right| \leq C\left(\widetilde{K}_{B, S}^{(\tau)}\right)^{\gamma}.
\end{equation}
The $\widetilde{\mathcal{L}}^{\psi,\tau,\gamma}$ norm of $f$, or $\|f\|_{\widetilde{\mathcal{L}}^{\psi,\tau,\gamma}}$, is defined as the infimum of $C>0$ satisfying (6) and (7).
\end{definition}

\begin{remark}
The following conclusions show that $\widetilde{\mathcal{L}}^{\psi,\tau,\gamma}$ is the generalization of $\mathrm{\widetilde{RBMO}(\mu)}$, the Campanato space over $(\mathcal{X},d,\mu)$, and the generalized Campanato space over $\mathbb{R}^n$.

$\mathrm{(i)}$ If $\psi(x,r)=1$, then $\widetilde{\mathcal{L}}^{\psi,\tau,\gamma}=\mathrm{\widetilde{RBMO}(\mu)}$ defined in\cite{Fu14}.

$\mathrm{(ii)}$ By Remark 2.1, for $\alpha\in[0,\infty)$, $\psi(x,r)=\lambda(x,r)^{\alpha}$ satisfies (5), then $\widetilde{\mathcal{L}}^{\psi,\tau,\gamma}=\mathcal{E}^{\alpha,1}_{\tau,\tau,\gamma}$ defined in\cite{Fu15}.

$\mathrm{(iii)}$ If $(\mathcal{X},d,\mu)=(\mathbb{R}^n,|\cdot|,m_n)$, where $m_n$ denotes the $n$-dimensional Lebesgue measure, then $\widetilde{\mathcal{L}}^{\psi,1,\gamma}=\mathcal{L}^{1,\psi}(\mathbb{R}^n)$.
\end{remark}

\begin{remark}
We will prove that $\widetilde{\mathcal{L}}^{\psi,\tau,\gamma}$ is independent of $\tau$ and $\gamma$ under a certain condition.
\end{remark}

Moreover, we recall the definition of $\theta$-type generalized Calder\'{o}n-Zygmund kernel and Marcinkiewicz integral operator. 

\begin{definition}\cite{Yab85}
Let $l \geq 0$, $\theta:(0,\infty)\rightarrow[0,\infty)$ be non-decreasing and make
$$\int_{0}^{1} \frac{\theta(t)}{t} \log\frac{1}{t}d t<\infty.$$
$K_{l, \theta}\in L_{loc}^1$ defined on $\mathcal{X}^2\backslash\{(x, x): x \in \mathcal{X}\}$ is a $\theta$-type generalized Calder\'{o}n-Zygmund kernel, if for $x,y\in\mathcal{X}$,
$$|K_{l, \theta}(x, y)| \lesssim\frac{(d(x, y))^{1+l}}{\lambda(x, d(x, y))},$$
and for $x, y,z \in \mathcal{X}$ with $d(x, y) \geq d(x, z)/2$,
$$|K_{l, \theta}(x, y)-K_{l, \theta}(z, y)|-|K_{l, \theta}(y, x)-K_{l, \theta}(y, z)| \lesssim \theta\left(\frac{d(x, z)}{d(x, y)}\right) \frac{(d(x, z))^{1+l}}{\lambda(x, d(x, y))}.$$
\end{definition}

\begin{definition}\cite{Lu22}
Let $l \geqslant 0, \rho>0$ and $s \geqslant 1$, the fractional type Marcinkiewicz integral operator $\widetilde{\mathcal{M}}_{l, \rho, s}$ with $\theta$-type generalized Calder\'{o}n-Zygmund kernel $K_{l, \theta}$ is defined by
$$\widetilde{\mathcal{M}}_{l, \rho, s}(f)(x)=\left(\int_{0}^{+\infty}\left|\frac{1}{t^{l+\rho}} \int_{d(x, y) \leq t} \frac{K_{l, \theta}(x, y)}{(d(x, y))^{1-\rho}} f(y) d \mu(y)\right|^{s} \frac{d t}{t}\right)^{\frac{1}{s}}$$
for $f \in L_{c}^{\infty}(\mu)$, $x \notin \operatorname{supp}(f)$, and the commutator $\widetilde{\mathcal{M}}_{l, \rho, s, b}$ generated by $b \in \widetilde{\mathcal{L}}^{\psi,\tau,\gamma}$ and $\widetilde{\mathcal{M}}_{l, \rho, s}$ is defined by
$$\widetilde{\mathcal{M}}_{l, \rho, s, b}(f)(x)=\left(\int_{0}^{+\infty}\left|\frac{1}{t^{l+\rho}} \int_{d(x, y) \leq t} (b(x)-b(y))\frac{K_{l, \theta}(x, y)}{(d(x, y))^{1-\rho}} f(y) d \mu(y)\right|^{s} \frac{d t}{t}\right)^{\frac{1}{s}}$$
for $f \in L_{c}^{\infty}(\mu)$, $x \in \mathcal{X}$.
\end{definition}

\begin{remark}
If $(\mathcal{X},d,\mu)=(\mathbb{R}^n,|\cdot|,m_n)$, $l=0$, $\rho=1$, $K_{l,\theta}(x,y)=\displaystyle\frac{\Omega(x-y)}{|x-y|^{n-1}}$, then $\widetilde{\mathcal{M}}_{l, \rho, s}=\mathcal{M}_{\Omega}$ defined by Stein in\cite{Ste58}.
\end{remark}

The following two conditions will be used in some situations.

\begin{definition}\cite{Fu15,Fu142}
Let $\tau\in(1,\infty)$, $\mu$ satisfies the $\tau$-weak doubling condition, or $\mu\in \mathcal{D}_{\tau}$, if for all balls $B\subset\mathcal{X}$,
$$N^{(\tau)}_{B,\widetilde{B}^{\tau}}\leq C_{(\mu)}.$$
Let $\sigma \in(0, \infty)$, the function $\lambda$ defined in Definition 2.2 satisfies the $\sigma$-weak reverse doubling condition, or $\lambda\in\mathcal{R}_{\sigma}$, if for any $x \in \mathcal{X}$, $0<r<2 \operatorname{diam}(\mathcal{X})$ and $1<a<2 \operatorname{diam}(\mathcal{X}) / r$,
$$C_{(a)} \lambda(x, r)\leq \lambda(x, a r),$$
and
$$\sum_{j=1}^{\infty} \frac{1}{C_{(a^{j})}^{\sigma}}<\infty.$$
\end{definition}

To obtain the boundedness of $\widetilde{\mathcal{M}}_{l, \rho, s}$ and $\widetilde{\mathcal{M}}_{l, \rho, s,b}$, the following maximal operators are needed.

\begin{definition}\cite{Fu142}
The sharp maximal operator $\widetilde{M}^{\sharp}$ is defined as
$$\widetilde{M}^{\sharp} f(x)=\sup _{B \ni x} \frac{1}{\mu(6 B)} \int_{B}\left|f(y)-f_{B}\right| d \mu(y)+\sup _{(B, S) \in \Delta_{x}} \frac{\left|f_{B}-f_{S}\right|}{\widetilde{K}_{B, S}^{(6)}}$$
for any $f\in L^1_{loc}$ and $x \in \mathcal{X}$, where $\Delta_{x}$ is consisted of all pairs of doubling balls $(B,S)$ with $x\in B\subset S$.
\end{definition}

\begin{definition}\cite{Fu142}
Let $p \in(1, \infty)$ and $\tau \in[5, \infty)$, define
$$M_{p, \tau} f(x)=\sup _{B \ni x}\left(\frac{1}{\mu(\tau B)} \int_{B}|f(y)|^{p} d \mu(y)\right)^{\frac{1}{p}}$$
for any $f\in L^p_{loc}$, $x\in\mathcal{X}$, and
$$N f(x)=\sup _{\substack{doubling\ ball\\ B \ni x}} \fint_{B}|f(y)| d \mu(y)$$
for any $f\in L^1_{loc}$, $x \in \mathcal{X}$.
\end{definition}

\begin{definition}
Let $p \in(1, \infty)$, $\tau \in[5, \infty)$ and $\psi$ satisfy (5), define
$$M_{\psi,p, \tau} f(x)=\sup _{B \ni x}\psi(B)\left(\frac{1}{\mu(\tau B)} \int_{B}|f(y)|^{p} d \mu(y)\right)^{\frac{1}{p}}$$
for any $f\in L^p_{loc}$, $x\in\mathcal{X}$.
\end{definition}

In Section 4, we will use the following operator $T_{\lambda}$ to control $\widetilde{\mathcal{M}}_{l, \rho, s}$.

\begin{definition}
Let $\lambda$ be defined in Definition 2.2, define
$$T_{\lambda}(f)(x)=\int_{\mathcal{X}} \frac{f(y)}{\lambda(x, d(x, y))} d \mu(y)$$
for any $f \in L_{c}^{\infty}(\mu)$ and $x \notin \operatorname{supp}(f)$.
\end{definition}

Finally, we recall some lemmas about geometrically doubling metric space, discrete coefficients and several maximal operators, which will be used in Section 3 and Section 4.

\begin{lemma}\cite{Lin17}
The following propositions exist:

$\mathrm{(i)}$ There holds $\widetilde{K}_{B,R}^{(\tau)}\leq C_{(\tau)}\widetilde{K}_{B,S}^{(\tau)}$ for any $\tau\in(1,\infty)$ and balls $B\subset R\subset S$.

$\mathrm{(ii)}$ There holds $\widetilde{K}_{B,S}^{(\tau)}\leq C_{(\alpha,\tau)}$ for any $\alpha\in[1,\infty)$, $\tau\in(1,\infty)$ and balls $B\subset S$ with $r_s\leq\alpha r_B$.

$\mathrm{(iii)}$ There holds $\widetilde{K}^{(\tau)}_{B,S}\leq C_{(\tau,\beta,\nu)}$ for any $\tau,\eta,\beta\in(1,\infty)$ and concentric balls $B\subset S$, such that the $(\tau,\beta)$-doubling ball $\tau^kB$ with $k\in\mathbb{N}$ and satisfying $B\subset\tau^kB\subset S$ does not exist, where $\nu$ is defined in Remark 2.2.

$\mathrm{(iv)}$ For any $\tau\in(1,\infty)$ and balls $B\subset R\subset S$,
$$\widetilde{K}^{(\tau)}_{B,S}\leq\widetilde{K}^{(\tau)}_{B,R}+C_{(\tau,\nu)}\widetilde{K}^{(\tau)}_{R,S}.$$

$\mathrm{(v)}$ For any $\tau\in(1,\infty)$ and balls $B\subset R\subset S$, $\widetilde{K}^{(\tau)}_{R,S}\leq C_{(\tau,\nu)}\widetilde{K}^{(\tau)}_{B,S}$.

$\mathrm{(vi)}$ For any $\tau_1,\tau_2\in(1,\infty)$ and balls $B\subset S$,
$$c_{(\tau_1,\tau_2,\nu)}\widetilde{K}^{(\tau_1)}_{B,S}\leq\widetilde{K}^{(\tau_2)}_{B,S}\leq C_{(\tau_1,\tau_2,\nu)}\widetilde{K}^{(\tau_1)}_{B,S}.$$
\end{lemma}

\begin{lemma}\cite{Fu15}
Let $\tau>1$, $m>1$ be an integer, and $B_1\subset B_2\subset\cdots\subset B_m$ be concentric balls with radii $\tau^Nr_{B_1}$, where $N\in\mathbb{N}$. If $\widetilde{K}_{B_i,B_{i+1}}^{(\tau)}>3+\lfloor\log_{\tau}2\rfloor$ for any $i\in\{1,2,\cdots,m-1\}$, then,
$$\sum_{i=1}^{m-1}\widetilde{K}_{B_i,B_{i+1}}^{(\tau)}<(3+\lfloor\log_{\tau}2\rfloor)\widetilde{K}_{B_1,B_m}^{(\tau)}.$$
\end{lemma}

\begin{lemma}\cite{Fu142}
Let $p \in(1, \infty)$ and $\tau \in[5, \infty)$, then $M_{p, \tau}$ and $N$ are bounded on $L^{p}$.
\end{lemma}

\begin{lemma}\cite{Fu142}
Let $f\in L_{loc}^1$ (and satisfy $\int_{\mathcal{X}} f(x) d \mu(x)=0$ if $\mu(\mathcal{X})<\infty$), and $\inf \{1, N f\} \in L^{p}$ for some $1<p<\infty$. Then,
$$\|N f\|_{L^{p}} \lesssim\left\|\widetilde{M}^{\sharp} f\right\|_{L^{p}}.$$
\end{lemma}

The last two lemmas can be directly obtained by Lemma 2.1, and Lemma 2.2 with the similar method used in\cite[Lemma 2.7]{Hyt12}, respectively. We omit the proofs here.

\begin{lemma}
Let $k>1$ and $j \in \mathbb{N}$, for all balls $B\subset\mathcal{X}$,
$$\frac{1}{\psi(B)}\left|f_{k B}-f_{B}\right| \lesssim\|f\|_{\widetilde{\mathcal{L}}^{\psi,\tau,\gamma}},$$
and
$$\frac{1}{\psi(B)}\left|f_{k^{j} B}-f_{B}\right| \lesssim j\|f\|_{\widetilde{\mathcal{L}}^{\psi,\tau,\gamma}}.$$
\end{lemma}

\begin{lemma}
Let $\tau>1$, $\mu\in\mathcal{D}_{\tau}$, then there exists $C>0$ satisfying that: For any $x\in\mathcal{X}$ and balls $B,S$ with $x\in B\subset S$ and $\widetilde{K}_{B,S}^{(\tau)}\leq C$,
$$\frac{1}{\psi(B)}|f_B-f_S|\leq C_{(x)}\widetilde{K}_{B,S}^{(\tau)},$$
then for balls $B,S$ with $x\in B\subset S$,
$$\frac{1}{\psi(B)}|f_B-f_S|\leq CC_{(x)}\widetilde{K}_{B,S}^{(\tau)}.$$
\end{lemma}

\section{Properties and characterizations of $\widetilde{\mathcal{L}}^{\psi,\tau,\gamma}$} 

In this section, we first prove the independence between $\widetilde{\mathcal{L}}^{\psi,\tau,\gamma}$ and some of its parameters, then obtain the John-Nirenberg inequality on $\widetilde{\mathcal{L}}^{\psi,\tau,\gamma}$.

\begin{theorem}
$\widetilde{\mathcal{L}}^{\psi,\tau,\gamma}$ is independent of $\tau>1$.
\end{theorem}

\begin{proof}
Let $1<\tau_1<\tau_2$, by Lemma 2.1, for any balls $B\subset S$,
$$\left(\widetilde{K}_{B,S}^{(\tau_1)}\right)^{\gamma}\approx\left(\widetilde{K}_{B,S}^{(\tau_2)}\right)^{\gamma},$$
so only the condition (6) need to be considered. By $\mu(\tau_2B)\geq\mu(\tau_1B)$, $\widetilde{\mathcal{L}}^{\psi,\tau_1,\gamma}\subset\widetilde{\mathcal{L}}^{\psi,\tau_2,\gamma}$.\\
Conversely, assume that $f\in\widetilde{\mathcal{L}}^{\psi,\tau_2,\gamma}$, let $\delta=(\tau_1-1)/\tau_2$, for a fixed ball $B_0=B(x_0,r)$, by \cite[Lemma 2.3]{Hyt10}, there exists balls $B_i=B(x_i,\delta r)$ cover $B_0$, where $x_i\in B_0$, $i\in I$, and the number of elements in $I$ is not more than $N\delta^{-n}$. Since $r+\delta\tau_2r=\tau_1r$, $\tau_2B_i=B(x_i,\delta\tau_2r)\subset B(x_0,\tau_1r)=\tau_1B_0$, thus by Lemma 2.1 and Lemma 2.5,
$$\frac{1}{\psi(B_0)}|f_{B_i}-f_{B_0}|\leq\frac{1}{\psi(B_0)}(|f_{B_i}-f_{\tau_1B_0}|+|f_{\tau_1B_0}-f_{B_0}|)\lesssim\|f\|_{\widetilde{\mathcal{L}}^{\psi,\tau_2,\gamma}},$$
therefore, by Lemma 2.1,
\begin{align*}
\int_{B_0}|f-f_{B_0}|d\mu&\leq\sum_{i\in I}\int_{B_i}|f-f_{B_0}|d\mu\leq\sum_{i\in I}\left(\int_{B_i}|f-f_{B_i}|d\mu+|f_{B_i}-f_{B_0}|\mu(B_i)\right)\\
&\lesssim\sum_{i\in I}\psi(B_0)\|f\|_{\widetilde{\mathcal{L}}^{\psi,\tau_2,\gamma}}\mu(\tau_2B_i)\lesssim\|f\|_{\widetilde{\mathcal{L}}^{\psi,\tau_2,\gamma}}\psi(B_0)\mu(\tau_1B_0),
\end{align*}
hence $\|f\|_{\widetilde{\mathcal{L}}^{\psi,\tau_1,\gamma}}\lesssim\|f\|_{\widetilde{\mathcal{L}}^{\psi,\tau_2,\gamma}}$, which implies that $\widetilde{\mathcal{L}}^{\psi,\tau_2,\gamma}\subset\widetilde{\mathcal{L}}^{\psi,\tau_1,\gamma}$.
\end{proof}

\begin{theorem}
Let $\tau>1$, $\mu\in\mathcal{D}_{\tau}$, then $\widetilde{\mathcal{L}}^{\psi,\tau,\gamma}$ is independent of $\gamma\geq1$.
\end{theorem}

\begin{proof}
Since $\widetilde{K}_{B,S}^{(\tau)}\geq1$, for $\gamma\geq1$, $\widetilde{\mathcal{L}}^{\psi,\tau,1}\subset\widetilde{\mathcal{L}}^{\psi,\tau,\gamma}$. Conversely, assume that $f\in\widetilde{\mathcal{L}}^{\psi,\tau,\gamma}$, then for $x\in B\subset S$ such that $\widetilde{K}_{B,S}^{(\tau)}\leq C$, we have
$$\frac{1}{\psi(B)}|f_B-f_S|\leq\left(\widetilde{K}_{B,S}^{(\tau)}\right)^{\gamma}\|f\|_{\widetilde{\mathcal{L}}^{\psi,\tau,\gamma}}\leq C^{\gamma-1}\widetilde{K}_{B,S}^{(\tau)}\|f\|_{\widetilde{\mathcal{L}}^{\psi,\tau,\gamma}}.$$
Therefore, by Lemma 2.6, for any balls $B\subset S$,
$$\frac{1}{\psi(B)}|f_B-f_S|\lesssim C^{\gamma-1}\widetilde{K}_{B,S}^{(\tau)}\|f\|_{\widetilde{\mathcal{L}}^{\psi,\tau,\gamma}},$$
which implies that $\|f\|_{\widetilde{\mathcal{L}}^{\psi,\tau,1}}\lesssim \|f\|_{\widetilde{\mathcal{L}}^{\psi,\tau,\gamma}}$, thus $\widetilde{\mathcal{L}}^{\psi,\tau,\gamma}\subset\widetilde{\mathcal{L}}^{\psi,\tau,1}$.
\end{proof}

Even if the $\tau$-weak doubling condition does not hold, we still consider $\gamma=1$ in general. By Theorem 3.1, the space $\widetilde{\mathcal{L}}^{\psi,\tau,1}$ can be written as $\widetilde{\mathcal{L}}^{\psi}$.

The proof of the John-Nirenberg inequality needs some lemmas.

\begin{lemma}
Let $\alpha>1$, for all balls $B\subset\mathcal{X}$, there holds $\widetilde{K}_{B,\widetilde{B}^{\alpha}}^{(\alpha)}\leq C$, where $\widetilde{B}^{\alpha}$ is defined in Remark 2.2.
\end{lemma}

\begin{proof}
Set $\gamma=C_{(\lambda)}^{\log_2\alpha}$, then $\beta>\gamma$, and
$$\lambda(x,\alpha^kr)\leq C_{(\lambda)}\gamma^k\lambda(x,r).$$
Let $\widetilde{B}^{\alpha}=\alpha^jB$, by the minimality of $j$,
$$\widetilde{K}_{B,\widetilde{B}^{\alpha}}^{(\alpha)}\leq1+\sum_{k=-\lfloor\log_{\alpha}2\rfloor}^{j}\frac{\mu(\alpha^kB)}{\lambda(c_B,\alpha^kr_B)}\lesssim1+\sum_{k=1}^{j}\frac{\beta^{k-j}\mu(\alpha^jB)}{\gamma^{k-j}\lambda(c_B,\alpha^jr_B)}\lesssim1+\sum_{k=1}^{j}\left(\frac{\gamma}{\beta}\right)^{j-k}\leq C,$$
which completes the proof.
\end{proof}

\begin{lemma}
Let $C_1,C_2>0$, $f\in\widetilde{\mathcal{L}}^{\psi}$, for all balls $B_1=B(x_1,r_1),B_2=B(x_2,r_2)$ with
$$C_1d(x_1,x_2)\leq \max(r_1,r_2)\leq C_2d(x_1,x_2),$$
there holds
$$\frac{1}{\psi(B_1)}|f_{B_1}-f_{B_2}|\lesssim\lVert f\rVert_{\widetilde{\mathcal{L}}^{\psi}}.$$
\end{lemma}

\begin{proof}
By the assumption of this lemma, there exists $m,M>0$ such that $B_1\cup B_2\subset mB_1$, and $mB_1\subset MB_2$. Then,
$$\frac{1}{\psi(B_1)}|f_{B_1}-f_{B_2}|\leq\frac{1}{\psi(B_1)}|f_{B_1}-f_{mB_1}|+\frac{1}{\psi(B_1)}|f_{mB_1}-f_{B_2}|.$$
By Lemma 2.5,
$$\frac{1}{\psi(B_1)}|f_{B_1}-f_{mB_1}|\lesssim\|f\|_{\widetilde{\mathcal{L}}^{\psi}},$$
and
$$\frac{1}{\psi(B_1)}|f_{mB_1}-f_{B_2}|\lesssim\widetilde{K}^{(\tau)}_{B_2,mB_1}\|f\|_{\widetilde{\mathcal{L}}^{\psi}}\leq\left(1+\sum_{k=-\lfloor\log_{\tau}2\rfloor}^{N^{(\tau)}_{B_2,mB_1}}\frac{\mu(\tau^kB_2)}{\lambda(c_{B_1},\tau^kmr_{B_1})}\right)\|f\|_{\widetilde{\mathcal{L}}^{\psi}}\lesssim\|f\|_{\widetilde{\mathcal{L}}^{\psi}},$$
which completes the proof.
\end{proof}

\begin{theorem}
Let $\tau>1$, $\mu\in\mathcal{D}_{\tau}$, then there exists $C>0$ such that, for any $f\in \widetilde{\mathcal{L}}^{\psi}$, $t>0$ and ball $B=B(x_0,r)\subset\mathcal{X}$,
$$\mu\left(\left\{x\in B:\frac{|f(x)-f_{B}|}{\psi(B)}>t\right\}\right)\leq2\exp\left(-\frac{Ct}{\|f\|_{\widetilde{\mathcal{L}}^{\psi}}}\right)\mu(\tau B).$$
\end{theorem}

\begin{proof}
Let $\alpha=5\tau$, $G>0$ will be determined later. By \cite[Corollary 3.6]{Hyt10}, for $\mu$-a.e. $x\in B$ with $|f(x)-f_{B}|/\psi(B)>G$, there exist $(\alpha,\beta)$-doubling balls $B(x,\alpha^{-i}r)$ for $i\in\mathbb{N}$ making
$$B(x,\alpha^{-i}r)\subset\sqrt{\tau}B\text{ \ and \ }\frac{|f_{B(x,\alpha^{-i}r)}-f_{B}|}{\psi(B)}>G.$$
Let $B'_x$ be the biggest ball satisfying such properties, then,
\begin{align*}
\frac{1}{\mu(B'_x)}\int_{B'_x}\frac{|f(y)-f_{B}|}{\psi(B)}d\mu(y)&\geq\frac{|f_{B'_x}-f_{B}|}{\psi(B)}-\frac{1}{\mu(B'_x)}\int_{B'_x}\frac{|f(y)-f_{B'_x}|}{\psi(B)}d\mu(y)\\
&>G-\frac{\psi(B'_x)}{\psi(B)}\beta\|f\|_{\widetilde{\mathcal{L}}^{\psi}}=G-C_1\|f\|_{\widetilde{\mathcal{L}}^{\psi}}\geq\frac{G}{2},
\end{align*}
provided that $G\geq 2C_1\|f\|_{\widetilde{\mathcal{L}}^{\psi}}$.\\
Denote $B''_x:=\widetilde{(\alpha B'_x)}^{\alpha}$, by the maximality of $B'_x$,
$$B''_x\not\subset\sqrt{\tau}B\text{ \ or \ }\frac{|f_{B''_x}-f_{B}|}{\psi(B)}\leq G.$$
Moreover, if $B''_x\not\subset\sqrt{\tau}B$, let $\alpha^jB'_x$ be the smallest ball $\alpha^kB'_x\ (k\in\mathbb{N})$ satisfying $\alpha^kB'_x\not\subset\sqrt{\tau}B$, there holds
$$r_{\alpha^jB'_x}\approx r_{B}\text{ \ and \ }\alpha^jB'_x\subset3\alpha\sqrt{\tau}B.$$
Therefore, by Lemma 2.5, Lemma 3.1 and Lemma 3.2,
\begin{align*}
\frac{|f_{B''_x}-f_{B}|}{\psi(B)}&\leq\frac{|f_{B''_x}-f_{\alpha^jB'_x}|}{\psi(B)}+\frac{|f_{\alpha^jB'_x}-f_{3\alpha\sqrt{\tau}B}|}{\psi(B)}+\frac{|f_{3\alpha\sqrt{\tau}B}-f_{B}|}{\psi(B)}\\
&\lesssim\frac{\psi(B''_x)}{\psi(B)}\|f\|_{\widetilde{\mathcal{L}}^{\psi}}+\frac{\psi(3\alpha\sqrt{\tau}B)}{\psi(B)}\left(\widetilde{K}_{\alpha^jB'_x,3\alpha\sqrt{\tau}B}^{(\tau)}+\widetilde{K}_{B,3\alpha\sqrt{\tau}B}^{(\tau)}\right)\|f\|_{\widetilde{\mathcal{L}}^{\psi}}\\
&\lesssim C_2\|f\|_{\widetilde{\mathcal{L}}^{\psi}}\leq G,
\end{align*}
provided that $G\geq C_2\|f\|_{\widetilde{\mathcal{L}}^{\psi}}$. Thus, if $G\geq C_2\|f\|_{\widetilde{\mathcal{L}}^{\psi}}$, there holds
$$\frac{|f_{B''_x}-f_{B}|}{\psi(B)}\leq G.$$
Furthermore, by \cite[Theorem 1.2]{Hei01} and\cite[Lemma 2.5]{Hyt10}, there exist pairwise disjoint balls $\{B'_{x_\lambda}\}_{\lambda\in \Lambda}$ satisfying $x_\lambda\in B$ for any $\lambda\in \Lambda$, and
$$B\subset\bigcup_{x\in B}B'_x\subset\bigcup_{\lambda\in \Lambda}5B'_{x_\lambda}.$$
Denote $B^{\lambda_1}=5B'_{x_\lambda}$ for $\lambda\in \Lambda$, then, for any integer $m>1$, if $x\in B$ and $|f(x)-f_{B}|/\psi(B)>mG$, there exists $\lambda\in \Lambda$ such that $x\in B^{\lambda_1}$. By Lemma 3.1,
\begin{align*}
\frac{|f(x)-f_{B^{\lambda_1}}|}{\psi(B)}&\geq\frac{|f(x)-f_{B}|}{\psi(B)}-\frac{|f_{B}-f_{B''_{x_\lambda}}|}{\psi(B)}-\frac{|f_{B''_{x_\lambda}}-f_{5B'_{x_\lambda}}|}{\psi(B)}\\
&>mG-G-\frac{\psi(B''_{x_\lambda})}{\psi(B)}\widetilde{K}_{5B'_{x_\lambda},B''_{x_\lambda}}^{(\tau)}\\
&\geq (m-1)G-C_3\|f\|_{\widetilde{\mathcal{L}}^{\psi}}\geq (m-2)G,
\end{align*}
provided that $G\geq C_3\|f\|_{\widetilde{\mathcal{L}}^{\psi}}$.\\
By Lemma 2.5, we further get
\begin{align*}
\sum_{\lambda\in \Lambda}\mu(\tau B^{\lambda_1})&=\sum_{\lambda\in \Lambda}\mu(\alpha B'_{x_\lambda})\leq\beta\sum_{\lambda\in \Lambda}\mu(B'_{x_\lambda})\leq\frac{2\beta}{G}\sum_{\lambda\in\Lambda}\int_{B'_{x_\lambda}}\frac{|f(y)-f_{B}|}{\psi(B)}d\mu(y)\\
&\leq\frac{2\beta}{G}\left(\int_{\sqrt{\tau}B}\frac{|f(y)-f_{\sqrt{\tau}B}|}{\psi(B)}d\mu(y)+\frac{|f_{\sqrt{\tau}{B}}-f_{B}|\mu(\sqrt{\tau}B)}{\psi(B)}\right)\\
&\lesssim\frac{1}{G}\frac{\psi(\sqrt{\tau}B)}{\psi(B)}\mu(\tau B)\|f\|_{\widetilde{\mathcal{L}}^{\psi}}\leq\frac{C_4}{G}\mu(\tau B)\|f\|_{\widetilde{\mathcal{L}}^{\psi}}\leq\frac{1}{2}\mu(\tau B),
\end{align*}
provided that $G\geq2C_4\|f\|_{\widetilde{\mathcal{L}}^{\psi}}$.\\
Replace $B$ with $B^{\lambda_1}$ and iteratively define $B^{\lambda_1,\lambda_2}$, $B^{\lambda_1,\lambda_2,\lambda_3}$ and so on. The similar calculation to above implies that
\begin{align*}
\left\{x\in B:\frac{|f(x)-f_{B}|}{\psi(B)}>2mG\right\}&\subset\bigcup_{\lambda_1}\left\{x\in B^{\lambda_1}:\frac{|f(x)-f_{B^{\lambda_1}}|}{\psi(B)}>2(m-1)G\right\}\\
&\subset\cdots\subset\bigcup_{\lambda_1,\lambda_2,\cdots,\lambda_m}\left\{x\in B^{\lambda_1,\lambda_2,\cdots,\lambda_m}:\frac{|f(x)-f_{B^{\lambda_1,\lambda_2,\cdots,\lambda_m}}|}{\psi(B)}>0\right\},
\end{align*}
therefore,
\begin{align*}
\mu\left(\left\{x\in B:\frac{|f(x)-f_{B}|}{\psi(B)}>2mG\right\}\right)&\leq\sum_{\lambda_1,\lambda_2,\cdots,\lambda_m}\mu(B^{\lambda_1,\lambda_2,\cdots,\lambda_m})\\
&\leq\frac{1}{2}\sum_{\lambda_1,\lambda_2,\cdots,\lambda_{m-1}}\mu(\tau B^{\lambda_1,\lambda_2,\cdots,\lambda_{m-1}})\\
&\leq\cdots\leq\frac{1}{2^m}\mu(\tau B).
\end{align*}
Take $G=C_0\|f\|_{\widetilde{\mathcal{L}}^{\psi}}$, and fix $n\in\mathbb{N}$ satisfying $t\in[2mG,2(m+1)G)$, then,
\begin{align*}
\mu\left(\left\{x\in B:\frac{|f(x)-f_{B}|}{\psi(B)}>t\right\}\right)&\leq\mu\left(\left\{x\in B:\frac{|f(x)-f_{B}|}{\psi(B)}>2mG\right\}\right)\\
&\leq\frac{1}{2^m}\mu(\tau B)\leq2\exp\left(-\frac{Ct}{\|f\|_{\widetilde{\mathcal{L}}^{\psi}}}\right)\mu(\tau B),
\end{align*}
which completes the proof.
\end{proof}

Finally, as an application of Theorem 3.3, the following equivalent characterization of generalized Campanato spaces can be obtained.

\begin{corollary}
Let $\tau>1$, $\mu\in\mathcal{D}_{\tau}$, $1<p<\infty$, if $f \in \widetilde{\mathcal{L}}^{\psi}$, then for any ball $B\subset\mathcal{X}$,
$$\frac{1}{\psi(B)}\left(\frac{1}{\mu(\tau B)} \int_B\left|f(x)-f_B\right|^{p} d \mu(x)\right)^{\frac{1}{p}}\approx\|f\|_{\widetilde{\mathcal{L}}^{\psi}}.$$
\end{corollary}

\begin{proof}
The ``$\gtrsim$'' part is directly obtained by the H\"{o}lder inequality. Conversely, by Theorem 3.3,
\begin{align*}
\int_B|f(x)-f_B|^pd\mu(x)&=\int_0^{\infty}pt^{p-1}\mu(\{x\in B:|f(x)-f_B|>t\})dt\\
&\lesssim\mu(\tau B)\int_0^{\infty}t^{p-1}\exp\left(-\frac{Ct\psi(B)}{\|f\|_{\widetilde{\mathcal{L}}^{\psi}}}\right)dt\\
&\approx\mu(\tau B)\left(\frac{\|f\|_{\widetilde{\mathcal{L}}^{\psi}}}{\psi(B)}\right)^p,
\end{align*}
which shows the ``$\lesssim$'' part.
\end{proof}

\section{The Boundedness of $\widetilde{\mathcal{M}}_{l, \rho, s}$ and $\widetilde{\mathcal{M}}_{l, \rho, s, b}$} 

By using the conclusions in Section 3, now we obtain the boundedness of $\widetilde{\mathcal{M}}_{l, \rho, s}$ and $\widetilde{\mathcal{M}}_{l, \rho, s, b}$ on generalized Morrey space.

\begin{lemma}
Let $1<p<\infty$, $0<\delta<1$, $\phi \in \mathcal{G}_{\delta}^{\text {dec }}$, $\lambda\in\mathcal{R}_{\sigma}$ for some $\sigma \in(0, \delta / p)$, and $T_{\lambda}$ be bounded on $L^{2}$, then $T_{\lambda}$ is bounded on $L^{p, \phi}$.
\end{lemma}

\begin{proof}
For any fixed ball $B$, decompose
$$f=f_{1}+f_{2}:=f \chi_{2 B}+f \chi_{(2 B)^c} .$$
Then write
\begin{align*}
&\left(\frac{1}{\phi(B) \mu(\eta B)} \int_{B}\left|T_{\lambda}(f)(x)\right|^{p} d \mu(x)\right)^{\frac{1}{p}}\\
&\ \ \ \ \leq\left(\frac{1}{\phi(B) \mu(\eta B)} \int_{B}\left|T_{\lambda}(f_{1})(x)\right|^{p} d \mu(x)\right)^{\frac{1}{p}}+\left(\frac{1}{\phi(B) \mu(\eta B)} \int_{B}\left|T_{\lambda}(f_{2})(x)\right|^{p} d \mu(x)\right)^{\frac{1}{p}} \\
&\ \ \ \ =:I_{1}+I_{2} .
\end{align*}
By the argument similar to \cite[Theorem 1]{Tao15}, $I_{1} \lesssim\|f\|_{L^{p, \phi}}$. For $I_{2}$, by the H\"{o}lder inequality,
\begin{align*}
\left|T_{\lambda}(f_{2})(x)\right| & \leq \int_{(2B)^c} \frac{|f(y)|}{\lambda(x, d(x, y))} d \mu(y) \\
&\lesssim \sum_{j=1}^{\infty} \frac{1}{\lambda\left(c_{B}, 2^{j+1} r_{B}\right)} \int_{2^{j+1} B}|f(y)| d \mu(y) \\
&\leq \sum_{j=1}^{\infty} \frac{1}{\lambda\left(c_{B}, 2^{j+1} r_{B}\right)}\left(\int_{2^{j+1} B}|f(y)|^{p} d \mu(y)\right)^{\frac{1}{p}}(\mu(2^{j+1} B))^{1-\frac{1}{p}} \\
&\leq\|f\|_{L^{p, \phi}} \sum_{j=1}^{\infty} \frac{(\phi(2^{j+1} B))^{\frac{1}{p}} \mu(2^{j+1} B)}{\lambda(c_{B}, 2^{j+1} r_{B})} \\
&\lesssim\|f\|_{L^{p, \phi}}(\phi(B))^{\frac{1}{p}}(\mu(B))^{\frac{\delta}{p}} \sum_{j=1}^{\infty} \frac{1}{(\lambda(c_{B}, 2^{j+1} r_{B}))^{\frac{\delta}{p}}} \\
&\lesssim\|f\|_{L^{p, \phi}}(\phi(B))^{\frac{1}{p}},
\end{align*}
which follows that $I_{2} \lesssim\|f\|_{L^{p, \phi}}$. Then we obtain that $\left\|T_{\lambda}(f)\right\|_{L^{p, \phi}} \lesssim\|f\|_{L^{p, \phi}}$.
\end{proof}

\begin{theorem}
Let $1<p<\infty$, $0<\delta<1$, $\phi \in \mathcal{G}_{\delta}^{\text {dec }}$, $\lambda\in\mathcal{R}_{\sigma}$ for some $\sigma \in(0, \delta / p)$, and $T_{\lambda}$ be bounded on $L^{2}$, then $\widetilde{\mathcal{M}}_{l, \rho, s}$ is bounded on $L^{p, \phi}$.
\end{theorem}

\begin{proof}
By the Minkowski inequality,
$$
\widetilde{\mathcal{M}}_{l, \rho, s}(f)(x) \leq \int_{\mathcal{X}} \frac{\left|K_{l, \theta}(x, y)\right|}{(d(x, y))^{1-\rho}}|f(y)|\left(\int_{d(x, y)}^{\infty} \frac{d t}{t^{1+(l+\rho) s}}\right)^{\frac{1}{s}} d \mu(y) \lesssim T_{\lambda}(|f|)(x),
$$
then by Lemma 4.1,
$$\left\|\widetilde{\mathcal{M}}_{l, \rho, s}(f)\right\|_{L^{p, \phi}} \lesssim\left\|T_{\lambda}(|f|)\right\|_{L^{p, \phi}} \lesssim\|f\|_{L^{p, \phi}},$$
which completes the proof.
\end{proof}

\begin{theorem}
Let $1<p<\infty$, $0<\delta<1$, $\phi \in \mathcal{G}_{\delta}^{\text {dec }}$, $\psi$ satisfy (5), $b \in\widetilde{\mathcal{L}}^{\psi}$, $\mu\in\mathcal{D}_{\tau}$, $\lambda\in\mathcal{R}_{\sigma}$ for some $\sigma \in(0, \delta / p)$, and $T_{\lambda}$ be bounded on $L^{2}$, then,
$$\widetilde{M}^{\sharp}(\widetilde{\mathcal{M}}_{l,\rho,s,b}(f))(x)\lesssim\|b\|_{\widetilde{\mathcal{L}}^{\psi}}\left(M_{\psi,p,5}(f)(x)+M_{\psi,p,6}(\widetilde{\mathcal{M}}_{l,\rho,s}(f))(x)\right).$$
\end{theorem}

\begin{proof}
By Definition 2.11, it suffices to show that, for all $x\in\mathcal{X}$ and balls $B\ni x$, 
\begin{equation}
\frac{1}{\mu(6 B)} \int_{B}\left|\widetilde{\mathcal{M}}_{l, \rho, s, b}(f)(y)-m_{B}\right| d \mu(y) \lesssim\|b\|_{\widetilde{\mathcal{L}}^{\psi}}\left(M_{\psi, p, 5}(f)(x)+M_{\psi, p, 6}(\widetilde{\mathcal{M}}_{l, \rho, s}(f))(x)\right),
\end{equation}
and for all doubling balls $B \subset S$ with $B \ni x$,
\begin{equation}
|m_B-m_S|\lesssim\widetilde{K}_{B,S}^{(6)}\|b\|_{\widetilde{\mathcal{L}}^{\psi}}\left(M_{\psi, p, 5}(f)(x)+M_{\psi, p, 6}(\widetilde{\mathcal{M}}_{l, \rho, s}(f))(x)\right),
\end{equation}
where
$$m_{B}:=m_{B}\left(\widetilde{\mathcal{M}}_{l, \rho, s}((b-b_{B}) f \chi_{\left(\frac{6}{5} B\right)^c})\right),\ \ m_{S}:=m_{S}\left(\widetilde{\mathcal{M}}_{l, \rho, s}((b-b_{S}) f \chi_{\left(\frac{6}{5} S\right)^c})\right).$$
To estimate (8), decompose
$$f=f_{1}+f_{2}:=f \chi_{\frac{6}{5} B}+f \chi_{\left(\frac{6}{5} B\right)^c}.$$
Then, write
\begin{align*}
&\frac{1}{\mu(6 B)} \int_{B}\left|\widetilde{\mathcal{M}}_{l, \rho, s, b}(f)(y)-m_{B}\right| d \mu(y) \\
&\ \ \ \ \leq \frac{1}{\mu(6 B)} \int_{B}\left|(b(y)-b_{B}) \widetilde{\mathcal{M}}_{l, \rho, s}(f)(y)\right| d \mu(y)+\frac{1}{\mu(6 B)} \int_{B}\left|\widetilde{\mathcal{M}}_{l, \rho, s}((b(\cdot)-b_{B}) f_{1})(y)\right| d \mu(y) \\
&\ \ \ \ \ \ \ \ +\frac{1}{\mu(6 B)} \int_{B}\left|\widetilde{\mathcal{M}}_{l, \rho, s}((b(\cdot)-b_{B}) f_{2})(y)-m_{B}\right| d \mu(y) \\
&\ \ \ \ =: D_{1}+D_{2}+D_{3}.
\end{align*}
By the H\"{o}lder inequality and Corollary 3.1,
\begin{align*}
D_{1} & \leq \frac{1}{\psi(B)}\left(\frac{1}{\mu(6 B)} \int_{B}\left|b(y)-b_{B}\right|^{p^{\prime}} d \mu(y)\right)^{\frac{1}{p^{\prime}}} \psi(B)\left(\frac{1}{\mu(6 B)} \int_{B}\left|\widetilde{\mathcal{M}}_{l, \rho, s}(f)(y)\right|^{p} d \mu(y)\right)^{\frac{1}{p}} \\
& \lesssim\|b\|_{\widetilde{\mathcal{L}}^{\psi}}\left(M_{\psi, p, 6}\left(\widetilde{\mathcal{M}}_{l, \rho, s}(f)\right)(x)\right) .
\end{align*}
By the H\"{o}lder inequality, Lemma 2.3, Corollary 3.1 and Lemma 2.5,
\begin{align*}
D_{2} & \leq \frac{1}{\mu(6 B)}\left(\int_{B}\left|\widetilde{\mathcal{M}}_{l, \rho, s}((b(\cdot)-b_{B}) f_{1})(y)\right|^{\sqrt{p}} d \mu(y)\right)^{\frac{1}{\sqrt{p}}}(\mu(B))^{1-\frac{1}{\sqrt{p}}} \\
& \lesssim\left(\frac{1}{\mu(6 B)} \int_{\frac{6}{5} B}\left|(b(y)-b_{B}) f(y)\right|^{\sqrt{p}} d \mu(y)\right)^{\frac{1}{\sqrt{p}}} \\
& \leq\left(\frac{1}{\mu(6 B)} \int_{\frac{6}{5} B}|f(y)|^{\sqrt{p} \sqrt{p}} d \mu(y)\right)^{\frac{1}{\sqrt{p}} \frac{1}{\sqrt{p}}}\left(\frac{1}{\mu(6 B)} \int_{\frac{6}{5} B}\left|b(y)-b_{B}\right|^{\sqrt{p}(\sqrt{p})^{\prime}} d \mu(y)\right)^{\frac{1}{\sqrt{p}} \frac{1}{(\sqrt{p})^{\prime}}}\\
&\leq \psi\left(\frac{6}{5} B\right)\left(\frac{1}{\mu(6 B)} \int_{\frac{6}{5} B}|f(y)|^{p} d \mu(y)\right)^{\frac{1}{p}} \\
&\ \ \ \times\frac{1}{\psi\left(\frac{6}{5} B\right)}\left[\left(\frac{1}{\mu(6 B)} \int_{\frac{6}{5} B}\left|b(y)-b_{\frac{6}{5} B}\right|^{\sqrt{p}(\sqrt{p})^{\prime}} d \mu(y)\right)^{\frac{1}{\sqrt{p}} \frac{1}{(\sqrt{p})^{\prime}}}+\left|b_{\frac{6}{5} B}-b_{B}\right|\right] \\
&\lesssim\|b\|_{\widetilde{\mathcal{L}}^{\psi}} M_{\psi, p, 5}(f)(x).
\end{align*}
Since
\begin{align*}
D_{3} & =\frac{1}{\mu(6 B)} \int_{B}\left|\widetilde{\mathcal{M}}_{l, \rho, s}((b(\cdot)-b_{B}) f_{2})(y)-\frac{1}{\mu(B)} \int_{B} \widetilde{\mathcal{M}}_{l, \rho, s}((b(\cdot)-b_{B}) f_{2})(z) d \mu(z)\right| d \mu(y) \\
& \leq \frac{1}{\mu(6 B)} \frac{1}{\mu(B)} \int_{B} \int_{B}\left|\widetilde{\mathcal{M}}_{l, \rho, s}((b(\cdot)-b_{B}) f_{2})(y)-\widetilde{\mathcal{M}}_{l, \rho, s}((b(\cdot)-b_{B}) f_{2})(z)\right| d \mu(y) d \mu(z),
\end{align*}
in order to estimate $D_{3}$, we estimate
$$E:=\left|\widetilde{\mathcal{M}}_{l, \rho, s}((b(\cdot)-b_{B}) f_{2})(y)-\widetilde{\mathcal{M}}_{l, \rho, s}((b(\cdot)-b_{B}) f_{2})(z)\right|.$$
By the Minkowski inequality, write
\begin{align*}
E&=\left|\left(\int_{0}^{+\infty}\left|\frac{1}{t^{l+\rho}} \int_{d(y, w) \leq t} (b(w)-b_{B})\frac{K_{l, \theta}(y, w)}{(d(y, w))^{1-\rho}} f_{2}(w) d \mu(w)\right|^{s} \frac{d t}{t}\right)^{\frac{1}{s}}\right.\\
&\ \ \ -\left.\left(\int_{0}^{+\infty}\left|\frac{1}{t^{l+\rho}} \int_{d(z, w) \leq t}(b(w)-b_{B}) \frac{K_{l, \theta}(z, w)}{(d(z, w))^{1-\rho}} f_{2}(w) d \mu(w)\right|^{s} \frac{d t}{t}\right)^{\frac{1}{s}} \right| \\
&\leq\left(\int_{0}^{+\infty}\left| \int_{d(y, w) \leq t} (b(w)-b_{B})\frac{K_{l, \theta}(y, w)}{(d(y, w))^{1-\rho}} f_{2}(w) d \mu(w)\right.\right. \\
&\ \ \ -\left.\left.\int_{d(z, w) \leq t}(b(w)-b_{B}) \frac{K_{l, \theta}(z, w)}{(d(z, w))^{1-\rho}} f_{2}(w) d \mu(w)\right|^{s} \frac{d t}{t^{1+(l+\rho) s}}\right)^{\frac{1}{s}} \\
&\leq\left(\int_{0}^{+\infty}\left|\int_{d(y, w) \leq t} (b(w)-b_{B})\frac{K_{l, \theta}(y, w)-K_{l, \theta}(z, w)}{(d(y, w))^{1-\rho}} f_{2}(w) d \mu(w)\right|^{s} \frac{d t}{t^{1+(l+\rho) s}}\right)^{\frac{1}{s}} \\
&\ \ \ +\left(\int_{0}^{+\infty}\left|\int_{d(y, w) \leq t<d(z, w)} (b(w)-b_{B})\frac{K_{l, \theta}(z, w)}{(d(y, w))^{1-\rho}} f_{2}(w) d \mu(w)\right|^{s} \frac{d t}{t^{1+(l+\rho) s}}\right)^{\frac{1}{s}} \\
&\ \ \ +\left(\int_{0}^{+\infty}\left|\int_{d(z, w) \leq t}(b(w)-b_{B})\left(\frac{K_{l, \theta}(z, w)}{(d(y, w))^{1-\rho}}-\frac{K_{l, \theta}(z, w)}{(d(z, w))^{1-\rho}}\right) f_{2}(w) d \mu(w)\right|^{s} \frac{d t}{t^{1+(l+\rho) s}}\right)^{\frac{1}{s}} \\
&=:E_{1}+E_{2}+E_{3}.
\end{align*}
For any $y, z \in B$, by the Minkowski inequality, the H\"{o}lder inequality, Corollary 3.1 and Lemma 2.5,
\begin{align*}
E_{1}&\leq \int_{\left(\frac{6}{5} B\right)^c} |b(w)-b_{B}|\frac{\left|K_{l, \theta}(y, w)-K_{l, \theta}(z, w)\right|}{(d(y, w))^{1-\rho}}|f(w)|\left(\int_{d(y, w)}^{+\infty} \frac{d t}{t^{1+(l+\rho) s}}\right)^{\frac{1}{s}} d \mu(w) \\
&\lesssim \sum_{j=1}^{\infty} \int_{\left(\frac{6}{5}\right)^{j+1} B \backslash\left(\frac{6}{5}\right)^{j} B} \theta\left(\frac{d(y, z)}{d\left(c_{B}, w\right)}\right)\left(\frac{d(y, z)}{d\left(c_{B}, w\right)}\right)^{1+l} \frac{|b(w)-b_{B}||f(w)|}{\lambda\left(c_{B}, d(y, w)\right)} d \mu(w) \\
&\lesssim \sum_{j=1}^{\infty}\left(\frac{6}{5}\right)^{-j(1+l)} \theta\left(\frac{1}{(\frac{6}{5})^{j}}\right) \frac{1}{\lambda\left(c_{B},(\frac{6}{5})^{j} r_{B}\right)} \int_{\left(\frac{6}{5}\right)^{j+1} B}\left|b(w)-b_{B} \| f(w)\right| d \mu(w) \\
&\leq \sum_{j=1}^{\infty}\left(\frac{6}{5}\right)^{-j(1+l)} \theta\left(\frac{1}{(\frac{6}{5})^{j}}\right) \frac{1}{\lambda\left(c_{B},(\frac{6}{5})^{j} r_{B}\right)}\left(\left|b_{\left(\frac{6}{5}\right)^{j+1} B}-b_{B}\right| \int_{\left(\frac{6}{5}\right)^{j+1} B}|f(w)| d \mu(w)\right.\\
&\ \ \ +\left.\int_{\left(\frac{6}{5}\right)^{j+1} B}\left|b(w)-b_{\left(\frac{6}{5}\right)^{j+1} B}\right||f(w)| d \mu(w)\right) \\
&\leq \sum_{j=1}^{\infty}\left(\frac{6}{5}\right)^{-j(1+l)} \theta\left(\frac{1}{(\frac{6}{5})^{j}}\right) \frac{1}{\lambda\left(c_{B},(\frac{6}{5})^{j} r_{B}\right)}\left[\frac{1}{\psi((\frac{6}{5})^{j+1} B)}\left|b_{\left(\frac{6}{5}\right)^{j+1} B}-b_{B}\right|\right. \\
&\ \ \ \times \psi\left((\frac{6}{5})^{j+1} B\right)\left(\int_{\left(\frac{6}{5}\right)^{j+1} B}|f(w)|^{p} d \mu(w)\right)^{\frac{1}{p}}\left(\mu\left((\frac{6}{5})^{j+1} B\right)\right)^{1-\frac{1}{p}}+\psi\left((\frac{6}{5})^{j+1} B\right) \\
&\ \ \ \times\left.\left(\int_{\left(\frac{6}{5}\right)^{j+1} B}|f(w)|^{p} d \mu(w)\right)^{\frac{1}{p}} \frac{1}{\psi((\frac{6}{5})^{j+1} B)}\left(\int_{\left(\frac{6}{5}\right)^{j+1} B}\left|b(w)-b_{\left(\frac{6}{5}\right)^{j+1} B}\right|^{p^{\prime}} d \mu(w)\right)^{\frac{1}{p'}}\right] \\
&\lesssim \sum_{j=1}^{\infty}\left(\frac{6}{5}\right)^{-j(1+l)} \theta\left(\frac{1}{(\frac{6}{5})^{j}}\right) \frac{1}{\lambda\left(c_{B},(\frac{6}{5})^{j} r_{B}\right)}\left[j\|b\|_{\widetilde{\mathcal{L}}^{\psi}}\left(\mu\left(5 \times(\frac{6}{5})^{j+1} B\right)\right)^{-\frac{1}{p}}\right. \\
&\ \ \ \times\psi\left((\frac{6}{5})^{j+1} B\right)\left(\int_{\left(\frac{6}{5}\right)^{j+1} B}|f(w)|^{p} d \mu(w)\right)^{\frac{1}{p}}\left(\mu\left(5 \times(\frac{6}{5})^{j+1} B\right)\right)^{\frac{1}{p}}\left(\mu\left((\frac{6}{5})^{j+1} B\right)\right)^{1-\frac{1}{p}} \\
&\ \ \ +\psi\left((\frac{6}{5})^{j+1} B\right)\left(\mu\left(5 \times(\frac{6}{5})^{j+1} B\right)\right)^{-\frac{1}{p}}\left(\int_{\left(\frac{6}{5}\right)^{j+1} B}|f(w)|^{p} d \mu(w)\right)^{\frac{1}{p}} \mu\left(5 \times(\frac{6}{5})^{j+1} B\right) \\
&\ \ \ \times\left.\frac{1}{\psi\left((\frac{6}{5})^{j+1} B\right)}\left(\frac{1}{\mu\left(5 \times(\frac{6}{5})^{j+1} B\right)} \int_{\left(\frac{6}{5}\right)^{j+1} B}\left|b(w)-b_{\left(\frac{6}{5}\right)^{j+1} B}\right|^{p^{\prime}} d \mu(w)\right)^{\frac{1}{p'}}\right] \\
&\lesssim\|b\|_{\widetilde{\mathcal{L}}^{\psi}} M_{\psi, p, 5}(f)(x) \sum_{j=1}^{\infty} j\left(\frac{6}{5}\right)^{-j(1+l)} \theta\left(\frac{1}{(\frac{6}{5})^{j}}\right) \frac{\mu\left(5 \times(\frac{6}{5})^{j+1} B\right)}{\lambda\left(c_{B},(\frac{6}{5})^{j} r_{B}\right)} \\
&\lesssim\|b\|_{\widetilde{\mathcal{L}}^{\psi}} M_{\psi, p, 5}(f)(x) \sum_{j=1}^{\infty}\left(\frac{6}{5}\right)^{-j(1+l)} \int_{\left(\frac{6}{5}\right)^{-j}}^{\left(\frac{6}{5}\right)^{-j+1}} \theta\left(\frac{1}{(\frac{6}{5})^{j}}\right)\left|\log \left(\frac{6}{5}\right)^{-j}\right| \frac{d t}{t} \\
&\lesssim\|b\|_{\widetilde{\mathcal{L}}^{\psi}} M_{\psi, p, 5}(f)(x) \sum_{j=1}^{\infty}\left(\frac{6}{5}\right)^{-j(1+l)} \int_{0}^{1} \frac{\theta(t)}{t} \log\frac{1}{t} d t \\
&\lesssim\|b\|_{\widetilde{\mathcal{L}}^{\psi}} M_{\psi, p, 5}(f)(x),
\end{align*}
by the similar calculation to $E_1$, we also have
\begin{align*}
E_{2}&\lesssim \int_{\left(\frac{6}{5} B\right)^c} \frac{1}{\lambda(z, d(z, w))}|b(w)-b_{B}||f(w)|\left(\left(\frac{d(z, w)}{d(y, w)}\right)^{(l+\rho) s}-1\right)^{\frac{1}{s}} d \mu(w) \\
&\lesssim \int_{\left(\frac{6}{5} B\right)^c} \frac{1}{\lambda(z, d(z, w))}|b(w)-b_{B}|| f(w)|\left(\frac{d(y, z)}{d(y, w)}\right)^{\frac{1}{s}} d \mu(w) \\
&\lesssim\|b\|_{\widetilde{\mathcal{L}}^{\psi}} \sum_{j=1}^{\infty} \frac{\left(\frac{6}{5}\right)^{-\frac{j}{s}}}{\lambda\left(c_{B},(\frac{6}{5})^{j} r_{B}\right)} \psi\left((\frac{6}{5})^{j+1} B\right) \int_{\left(\frac{6}{5}\right)^{j+1} B}|f(w)| d \mu(w) \\
&\ \ \ +\sum_{j=1}^{\infty} \frac{\left(\frac{6}{5}\right)^{-\frac{j}{s}}}{\lambda\left(c_{B},(\frac{6}{5})^{j} r_{B}\right)} \int_{\left(\frac{6}{5}\right)^{j+1} B \backslash\left(\frac{6}{5}\right)^{j} B}\left|b(w)-b_{\frac{6}{5} B}\right||f(w)| d \mu(w) \\
&\leq\|b\|_{\widetilde{\mathcal{L}}^{\psi}} \sum_{j=1}^{\infty} \frac{\left(\frac{6}{5}\right)^{-\frac{j}{s}}}{\lambda\left(c_{B},(\frac{6}{5})^{j} r_{B}\right)} \psi\left((\frac{6}{5})^{j+1} B\right)\left(\int_{\left(\frac{6}{5}\right)^{j+1} B}|f(w)|^{p} d \mu(w)\right)^{\frac{1}{p}} \\
&\ \ \ \times\left(\mu\left((\frac{6}{5})^{j+1} B\right)\right)^{1-\frac{1}{p}}+\sum_{j=1}^{\infty} \frac{\left(\frac{6}{5}\right)^{-\frac{j}{s}}}{\lambda\left(c_{B},(\frac{6}{5})^{j} r_{B}\right)}\left|b_{\frac{6}{5} B}-b_{\left(\frac{6}{5}\right)^{j+1} B}\right| \int_{\left(\frac{6}{5}\right)^{j+1} B}|f(w)| d \mu(w) \\
&\ \ \ +\sum_{j=1}^{\infty} \frac{\left(\frac{6}{5}\right)^{-\frac{j}{s}}}{\lambda\left(c_{B},(\frac{6}{5})^{j} r_{B}\right)} \int_{\left(\frac{6}{5}\right)^{j+1} B}\left|b(w)-b_{\left(\frac{6}{5}\right)^{j+1} B}\right||f(w)| d \mu(w) \\
&\lesssim\|b\|_{\widetilde{M}_{\psi}} M_{\psi, p, 5}(f)(x) \sum_{j=1}^{\infty}\left(\frac{6}{5}\right)^{-\frac{j}{s}} \frac{\mu\left(5 \times(\frac{6}{5})^{j+1} B\right)}{\lambda\left(c_{B},(\frac{6}{5})^{j} r_{B}\right)}\\
&\ \ \ +\|b\|_{\widetilde{\mathcal{L}}^{\psi}}\sum_{j=1}^{\infty} \frac{j\left(\frac{6}{5}\right)^{-\frac{j}{s}}}{\lambda\left(c_{B},(\frac{6}{5})^{j} r_{B}\right)} \psi\left((\frac{6}{5})^{j+1} B\right)\left(\int_{\left(\frac{6}{5}\right)^{j+1} B}|f(w)|^{p} d \mu(w)\right)^{\frac{1}{p}}\left(\mu\left((\frac{6}{5})^{j+1} B\right)\right)^{1-\frac{1}{p}}\\
&\ \ \ +\sum_{j=1}^{\infty} \frac{\left(\frac{6}{5}\right)^{-\frac{2}{s}}}{\lambda\left(c_{B},(\frac{6}{5})^{j} r_{B}\right)} \psi\left((\frac{6}{5})^{j+1} B\right)\left(\frac{1}{\mu\left(5 \times(\frac{6}{5})^{j+1} B\right)} \int_{\left(\frac{6}{5}\right)^{j+1} B}|f(w)|^{p} d \mu(w)\right)^{\frac{1}{p}}\\
&\ \ \ \times\frac{1}{\psi\left((\frac{6}{5})^{j+1} B\right)}\left(\frac{1}{\mu\left(5 \times(\frac{6}{5})^{j+1} B\right)} \int_{\left(\frac{6}{5}\right)^{j+1} B}\left| b(w)-b_{\left(\frac{6}{5}\right)^{j+1}B}\right|^{p'} d \mu(w)\right)^{\frac{1}{p^{\prime}}} \mu\left(5 \times(\frac{6}{5})^{j+1} B\right) \\
&\lesssim\|b\|_{\widetilde{\mathcal{L}}^{\psi}} M_{\psi, p, 5}(f)(x) \sum_{j=1}^{\infty} j\left(\frac{6}{5}\right)^{-\frac{j}{s}} \frac{\mu\left(5 \times(\frac{6}{5})^{j+1} B\right)}{\lambda\left(c_{B},(\frac{6}{5})^{j} r_{B}\right)} \\
&\lesssim\|b\|_{\widetilde{\mathcal{L}}^{\psi}} M_{\psi, p, 5}(f)(x),
\end{align*}
and,
\begin{align*}
E_{3} &\lesssim \int_{\left(\frac{6}{5} B\right)^c} \frac{1}{\lambda(z, d(z, w))} \frac{d(y, z)}{d(z, w)}|b(w)-b_{B}||f(w)| d \mu(w) \\
& \lesssim\|b\|_{\widetilde{\mathcal{L}}^{\psi}} M_{\psi, p, 5}(f)(x) \sum_{j=1}^{\infty} j\left(\frac{6}{5}\right)^{-j} \frac{\mu\left(5 \times(\frac{6}{5})^{j+1} B\right)}{\lambda\left(c_{B},(\frac{6}{5})^{j} r_{B}\right)} \\
& \lesssim\|b\|_{\widetilde{\mathcal{L}}^{\psi}} M_{\psi, p, 5}(f)(x),
\end{align*}
which, together with above estimates for $D_{1}$ and $D_{2}$, imply (8).\\
Then we show (9). Let $N_0:=N_{B, S}^{(6)}+1$, we have
\begin{align*}
|m_{B}-m_{S}|&\leq\left|m_{B}\left(\widetilde{\mathcal{M}}_{l, \rho, s}((b-b_{B}) f \chi_{(6^{N_0}B)^c})\right)-m_{S}\left(\widetilde{\mathcal{M}}_{l, \rho, s}((b-b_{B}) f \chi_{(6^{N_0}B)^c})\right)\right| \\
&\ \ \ +\left|m_{S}\left(\widetilde{\mathcal{M}}_{l, \rho, s}((b-b_{S}) f \chi_{(6^{N_0}B)^c})\right)-m_{S}\left(\widetilde{\mathcal{M}}_{l, \rho, s}((b-b_{B}) f \chi_{(6^{N_0}B)^c})\right)\right| \\
&\ \ \ +\left|m_{B}\left(\widetilde{\mathcal{M}}_{l, \rho, s}((b-b_{B}) f \chi_{(6^{N_0}B)\backslash\left(\frac{6}{5} B\right)})\right)\right|+\left|m_{S}\left(\widetilde{\mathcal{M}}_{l, \rho, s}((b-b_{S}) f \chi_{(6^{N_0}B)\backslash\left(\frac{6}{5} S\right)})\right)\right| \\
&=: F_{1}+F_{2}+F_{3}+F_{4}.
\end{align*}
By a similar argument to $D_{3}$,
$$F_{1} \lesssim \|b\|_{\widetilde{\mathcal{L}}^{\psi}}M_{\psi, p, 5}(f)(x),$$
and
$$F_{2} \lesssim  \widetilde{K}_{B, S}^{(6)}\|b\|_{\widetilde{\mathcal{L}}^{\psi}}M_{\psi, p, 6}\left(\widetilde{\mathcal{M}}_{l, \rho, s}(f)\right)(x).$$
For $y \in B$, by the Minkowski inequality, Lemma 2.5, the H\"{o}lder inequality and Corollary 3.1,
\begin{align*}
G&:=\widetilde{\mathcal{M}}_{l, \rho, s}\left((b-b_{B}) f \chi_{(6^{N_0}B)\backslash\left(\frac{6}{5} B\right)}(y)\right)\\
&=\left(\int_{0}^{+\infty}\left|\int_{d(y, w) \leq t} (b(w)-b_{B})\frac{K_{l, \theta}(y, w)}{(d(y, w))^{1-\rho}} f(w) \chi_{(6^{N_0}B)\backslash\left(\frac{6}{5} B\right)}(w) d \mu(w)\right|^{s} \frac{d t}{t^{1+(l+\rho) s}}\right)^{\frac{1}{s}} \\
&\leq \int_{(6^{N_0}B)\backslash\left(\frac{6}{5} B\right)} |b(w)-b_{B}|\frac{\left|K_{l, \theta}(y, w)\right|}{(d(y, w))^{1-\rho}}|f(w)|\left(\int_{d(y, w)}^{+\infty} \frac{d t}{t^{1+(l+\rho) s}}\right)^{\frac{1}{s}} d \mu(w) \\
&\lesssim \int_{(6^{N_0}B)\backslash\left(\frac{6}{5} B\right)} \frac{1}{\lambda(y, d(y, w))}|b(w)-b_{B}||f(w)| d \mu(w) \\
&\lesssim \sum_{j=1}^{N_0-1} \frac{1}{\lambda(c_{B}, 6^{j} r_{B})} \int_{6^{j+1} B}|b(w)-b_{B}|| f(w)| d \mu(w)+\frac{1}{\lambda(c_{B}, \frac{6}{5} r_{B})} \int_{6 B}|b(w)-b_{B}||f(w)| d \mu(w) \\
&\leq \sum_{j=1}^{N_0-1} \frac{1}{\lambda(c_{B}, 6^{j} r_{B})}\left(|b_{6^{j+1} B}-b_{B}| \int_{6^{j+1} B}|f(w)| d \mu(w)+\int_{6^{j+1} B}|b(w)-b_{6^{j+1} B}||f(w)| d \mu(w)\right) \\
&\ \ \ +\frac{1}{\lambda(c_{B}, \frac{6}{5} r_{B})}\left(|b_{6 B}-b_{B}| \int_{6 B}|f(w)| d \mu(w)+\int_{6 B}|b(w)-b_{6 B}||f(w)| d \mu(w)\right) \\
&\lesssim \sum_{j=1}^{N_0-1} \frac{1}{\lambda(c_{B}, 6^{j} r_{B})}\left[j\|b\|_{\widetilde{\mathcal{L}}^{\psi}}\psi(6^{j+1} B)\left(\int_{6^{j+1} B}|f(w)|^{p} d \mu(w)\right)^{\frac{1}{p}}(\mu(6^{j+1} B))^{1-\frac{1}{p}}\right. \\
&\ \ \ +\left.\left(\int_{6^{j+1} B}|f(w)|^{p} d \mu(w)\right)^{\frac{1}{p}}\left(\int_{6^{j+1} B}|b(w)-b_{6^{j+1} B}|^{p^{\prime}} d \mu(w)\right)^{\frac{1}{p'}}\right] \\
&\ \ \ +\frac{1}{\lambda(c_{B}, \frac{6}{5} r_{B})}\left[\|b\|_{\widetilde{\mathcal{L}}^{\psi}} \psi(6 B)\left(\int_{6 B}|f(w)|^{p} d \mu(w)\right)^{\frac{1}{p}}(\mu(6 B))^{1-\frac{1}{p}}\right. \\
&\ \ \ +\left.\left(\int_{6 B}|f(w)|^{p} d \mu(w)\right)^{\frac{1}{p}}\left(\int_{6 B}|b(w)-b_{6 B}|^{p^{\prime}} d \mu(w)\right)^{\frac{1}{p^{\prime}}}\right] \\
&\lesssim\|b\|_{\widetilde{\mathcal{L}}^{\psi}} M_{\psi, p, 5}(f)(x) \sum_{j=1}^{N_0-1} \frac{\mu(6^{j+1} B)}{\lambda(c_{B}, 6^{j} r_{B})}+\|b\|_{\widetilde{\mathcal{L}}^{\psi}} M_{\psi, p, 5}(f)(x) \frac{\mu(6 B)}{\lambda(c_{B}, \frac{6}{5} r_{B})} \\
&\lesssim\widetilde{K}_{B, S}^{(6)}\|b\|_{\widetilde{\mathcal{L}}^{\psi}} M_{\psi, p, 5}(f)(x),
\end{align*}
taking the mean over ball $B$, there exists
$$F_{3} \lesssim\widetilde{K}_{B, S}^{(6)}\|b\|_{\widetilde{\mathcal{L}}^{\psi}} M_{\psi, p, 5}(f)(x).$$
By a similar argument to $F_{3}$,
$$F_{4} \lesssim\widetilde{K}_{B, S}^{(6)}\|b\|_{\widetilde{\mathcal{L}}^{\psi}} M_{\psi, p, 5}(f)(x),$$
which combining above estimates with $F_{1}, F_{2}$ and $F_{3}$, yields (9).
\end{proof}

\begin{theorem}
Let $1<p\leq q<\infty$, $0<\delta<1$, $\phi \in \mathcal{G}_{\delta}^{\text {dec }}$, $\psi$ satisfy
\begin{equation}
\psi(B) \phi(B)^{\frac{1}{p}} \lesssim \phi(B)^{\frac{1}{q}}
\end{equation}
for all balls $B\subset\mathcal{X}$, then $M_{\psi, p, \tau}$ is bounded from $L^{p, \phi}$ to $L^{q, \phi}$.
\end{theorem}

\begin{proof}
Without the loss of generality, assume $\|f\|_{L^{p, \phi}}=1$, we firstly prove
\begin{equation}
M_{\psi, p, \tau}(f)(x) \lesssim M_{p, \tau}(f)(x)^{\frac{p}{q}}.
\end{equation}
That is, for any $B=B(y, r)\ni x$,
\[\psi(B)\left(\frac{1}{\mu(\tau B)} \int_{B}|f(y)|^{p} d \mu(y)\right)^{\frac{1}{p}} \lesssim M_{p, \tau}(f)(x)^{\frac{p}{q}} .\]
Fix $u>0$ which makes $\phi(y, u)=M_{p, \tau}(f)(x)^{p}$. If $u\geq r$, $\phi(B)=\phi(y, r) \geq M_{p, \tau}(f)(x)^{p}$, and $\phi(B)^{\frac{1}{q}-\frac{1}{p}} \leq M_{p, \tau}(f)(x)^{\frac{p}{q}-1}$. By (10),
\[\psi(B)\left(\frac{1}{\mu(\tau B)} \int_{B}|f(y)|^{p} d \mu(y)\right)^{\frac{1}{p}} \lesssim \phi(B)^{\frac{1}{q}-\frac{1}{p}} M_{p, \tau}(f)(x) \leq M_{p, \tau}(f)(x)^{\frac{p}{q}} .\]
If $u<r$, $\phi(B)=\phi(y, r) \leq M_{p, \tau}(f)(x)^{p}$, and $\phi(B)^{\frac{1}{q}} \leq M_{p, \tau}(f)(x)^{\frac{p}{q}}$. By $\|f\|_{L^{p, \phi}(\mu)}=1$ and (10),
\begin{align*}
\psi(B)\left(\frac{1}{\mu(\tau B)}\int_{B}|f(y)|^{p} d \mu(y)\right)^{\frac{1}{p}}&=\psi(B) \phi(B)^{\frac{1}{p}}\left(\frac{1}{\phi(B) \mu(\tau B)} \int_{B}|f(y)|^{p} d \mu(y)\right)^{\frac{1}{p}} \\
&\leq \psi(B) \phi(B)^{\frac{1}{p}}\lesssim \phi(B)^{\frac{1}{q}}\leq M_{p, \tau}(f)(x)^{\frac{p}{q}}.
\end{align*}
Then, by (11) and Lemma 2.3,
\begin{align*}
\left\|M_{\psi, p, \tau}(f)\right\|_{L^{q, \phi}}&=\sup _{B} \phi(B)^{-\frac{1}{q}} \mu(\tau B)^{-\frac{1}{q}}\left\|M_{\psi, p, \tau}(f)\right\|_{L^{q}(B)}\\
&\lesssim \sup _{B} \phi(B)^{-\frac{1}{q}} \mu(\tau B)^{-\frac{1}{q}}\left\|\left(M_{p, \tau}(f)\right)^{\frac{p}{q}}\right\|_{L^{q}(B)} \\
&=\sup _{B} \phi(B)^{-\frac{1}{q}} \mu(\tau B)^{-\frac{1}{q}}\left\|M_{p, \tau}(f)\right\|_{L^{p}(B)}^{\frac{p}{q}} \\
&\lesssim \sup _{B} \phi(B)^{-\frac{1}{q}} \mu(\tau B)^{-\frac{1}{q}}\|f\|_{L^{p}(B)}^{\frac{p}{q}} \\
&=\left(\sup _{B} \phi(B)^{-\frac{1}{p}} \mu(\tau B)^{-\frac{1}{p}}\|f\|_{L^{p}(B)}\right)^{\frac{p}{q}} \\
&=\|f\|_{L^{p, \phi}}^{\frac{p}{q}}=1,
\end{align*}
which completes the proof.
\end{proof}

Finally, we obtain the boundedness of commutator $\widetilde{\mathcal{M}}_{l, \rho, s, b}$ for $b\in\widetilde{\mathcal{L}}^{\psi}$. This result is even new when $\psi(x,r)=\lambda(x,r)^{\alpha}$ for $\alpha\geq0$, in which case $b$ is in Campanato space, see Remark 2.5.

\begin{theorem}
Let $1<p\leq q<\infty$, $0<\delta<1$, $\phi \in \mathcal{G}_{\delta}^{\text {dec }}$, $\psi$ satisfy (10), $b \in\widetilde{\mathcal{L}}^{\psi}$, $\mu\in\mathcal{D}_{\tau}$, $\lambda\in\mathcal{R}_{\sigma}$ for some $\sigma \in(0, \delta / p)$, and $T_{\lambda}$ be bounded on $L^{2}$, then  for $f\in L^{p,\phi}$,
$$\left\|\widetilde{\mathcal{M}}_{l, \rho, s, b}(f)\right\|_{L^{q, \phi}}\lesssim\|b\|_{\widetilde{\mathcal{L}}^{\psi}}\|f\|_{L^{p, \phi}}.$$
\end{theorem}

\begin{proof}
By Lemma 2.4, Theorem 4.2, Theorem 4.3 and Theorem 4.1,
\begin{align*}
&\left\|\widetilde{\mathcal{M}}_{l, \rho, s, b}(f)\right\|_{L^{q, \phi}}=\sup_B\phi(B)^{-\frac{1}{q}}\mu(\eta B)^{-\frac{1}{q}}\left\|\widetilde{\mathcal{M}}_{l, \rho, s, b}(f)\right\|_{L^q(B)}\\
&\ \ \ \ \lesssim\sup_B\phi(B)^{-\frac{1}{q}}\mu(\eta B)^{-\frac{1}{q}}\left\|N(\widetilde{\mathcal{M}}_{l, \rho, s, b}(f))\right\|_{L^q(B)}\\
&\ \ \ \ \lesssim\sup_B\phi(B)^{-\frac{1}{q}}\mu(\eta B)^{-\frac{1}{q}}\left\|\widetilde{M}^{\sharp}(\widetilde{\mathcal{M}}_{l, \rho, s, b}(f))\right\|_{L^q(B)}\\
&\ \ \ \ \lesssim\|b\|_{\widetilde{\mathcal{L}}^{\psi}}\sup_B\phi(B)^{-\frac{1}{q}}\mu(\eta B)^{-\frac{1}{q}}\left(\|M_{\psi,p,5}(f)\|_{L^q(B)}+\|M_{\psi,p,6}(\widetilde{\mathcal{M}}_{l,\rho,s}(f))\|_{L^q(B)}\right)\\
&\ \ \ \ \leq\|b\|_{\widetilde{\mathcal{L}}^{\psi}}\left(\|M_{\psi,p,5}(f)\|_{L^{q,\phi}}+\|M_{\psi,p,6}(\widetilde{\mathcal{M}}_{l,\rho,s}(f))\|_{L^{q,\phi}}\right)\\
&\ \ \ \ \lesssim\|b\|_{\widetilde{\mathcal{L}}^{\psi}}\left(\|f\|_{L^{p, \phi}}+\left\|\widetilde{\mathcal{M}}_{l, \rho, s}(f)\right\|_{L^{p, \phi}}\right)\lesssim\|b\|_{\widetilde{\mathcal{L}}^{\psi}}\|f\|_{L^{p, \phi}},
\end{align*}
which completes the proof.
\end{proof}

\bigskip \medskip\noindent
\textbf{\bf Acknowledgments}\\
The authors thank the referees for their careful reading and helpful comments which indeed improved the presentation of this article.\\
\textbf{\bf Funding information}\\
The research was supported by National Natural Science Foundation of China (Grant No. 12061069).\\
\textbf{\bf Authors contributions}\\
All authors have accepted responsibility for the entire content of this manuscript and approved its submission.\\
\textbf{\bf Conflict of interest}\\
Authors state no conflict of interest.
\bigskip \medskip

\noindent Yuxun Zhang and Jiang Zhou\\
\medskip
\noindent
College of Mathematics and System Sciences, Xinjiang University, Urumqi 830046\\
\smallskip
\noindent{E-mail }:
\texttt{zhangyuxun64@163.com} (Yuxun Zhang);
\texttt{zhoujiang@xju.edu.cn} (Jiang Zhou)

\bibliographystyle{plain}
\bibliography{Reference.bib}

\end{document}